\documentclass[10pt, reqno]{amsart} 

\usepackage{amsmath, amsthm, amssymb, amsfonts}
\usepackage{mathrsfs}
\usepackage{enumerate}
\usepackage{enumitem}
\usepackage{graphicx}
\usepackage{caption, subcaption}
\usepackage{float}
\usepackage{mathtools}

\usepackage{microtype}
\usepackage[colorlinks=true, linkcolor=blue, citecolor=red, urlcolor=blue]{hyperref}
\usepackage{orcidlink}

\usepackage{lineno} 

\usepackage[colorinlistoftodos, textsize=footnotesize]{todonotes}

\usepackage{breqn}

\theoremstyle{plain}
\newtheorem{theorem}{Theorem}[section]

\newtheorem{lemma}[theorem]{Lemma}

\theoremstyle{definition}

\theoremstyle{remark}

\allowdisplaybreaks

\title{Two absolutely bounded determinantal ratios}

\author{Hristo Sendov \orcidlink{0000-0002-0908-1535}}
\address{Department of Statistical and Actuarial Sciences \\
Department of Mathematics \\
Western University \\
1151 Richmond Street \\
London, ON, N6A 5B7 Canada}
\email{hsendov@uwo.ca}
\thanks{The first author was partially supported by the Natural Sciences and Engineering Research Council (NSERC) of Canada. (Grant number RGPIN-2020-06425.)}

\author{Mengxu Yuan \orcidlink{0009-0002-5166-183X}}
\address{Department of Mathematics \\
Western University \\
1151 Richmond Street \\
London, ON, N6A 5B7 Canada}
\email{myuan89@uwo.ca}


\begin{document}

\begin{abstract}
Bounded ratios of products of minors of positive definite matrices have a long history, starting with Hadamard's inequality in 1893. It states that for every positive semidefinite matrix $A$
$$
\det A \le A_{11} \cdots A_{nn}.
$$
This inequality was subsequently generalized by Fisher and then further by Koteljanskii. The latter states that for every positive semidefinite matrix $A$ and any index sets $\alpha_1, \alpha_2 \subseteq \{1,\ldots, n\}$ one has
$$
\det A[\alpha_1 \cup \alpha_2] \det A[\alpha_1 \cap \alpha_2] \le \det A[\alpha_1] \det A[\alpha_2],
$$
where $A[\alpha]$ denotes the principal submatrix determined by the indexes in $\alpha$.

In a manuscript published only on the arXiv in 2008, Hall and Johnson made three conjectures about ratios of products of principal minors of  \(4\times4\) positive definite matrices, denoted by $R_i$, $i=1,2,3$, see \eqref{2026-06-19-R1} and \eqref{2026-06-19-Ri}. They hypothesized that the supremum of $R_1$ was $27/16$, while the supremum of the other two ratios was $1$. Such ratios are called {\it absolutely bounded}.  The conjecture for $R_1$ was affirmed in \cite{sendov2026} and it is the only known bounded determinantal ratio with supremum bigger than one.  The goal of this paper is to affirm the conjecture for $R_2$ and $R_3$.

It is known that the upper bound for the ratios $R_i$, $i=1,2,3$, does not follow from repeated applications of Koteljanskii's inequality. In addition, Hall and Johnson showed that $R_i$ is bounded above by $4$, for $i=1,2,3$.
\end{abstract}

\maketitle

\section{Introduction}

Let $A := (a_{ij})$ denote an $n$-by-$n$ positive definite matrix. Consider a collection of index sets
$\alpha_i \subseteq \{1,\dots,n\}$, $i=1,\ldots, p$ and define $A[\alpha_i]$ to be the principal submatrix of $A$ at the intersection of the rows and columns indexed by $\alpha_i$. By convention, $A[\emptyset] := 1$.  Let $\alpha := \{\alpha_1,\ldots, \alpha_p\}$ and define the product of determinants 
$$
A(\alpha) := \prod_{i=1}^p \det A[\alpha_i].
$$
Much effort has been made to characterize the collection of index sets $\alpha$ and $\beta$ for which the 
ratio   $A(\alpha)/A(\beta)$ is bounded (more specifically, bounded by one) for all positive definite matrices $A$. Examples of these types of ratios are Hadamard's inequality and the more general Fisher's inequality. They are further generalized by Koteljanskii’s inequality
\begin{align}
\label{2022-07-27-ineq-3}
\frac{\det A[\alpha_1 \cup \alpha_2] \det A[\alpha_1 \cap \alpha_2] }{\det A[\alpha_1] \det A[\alpha_2]} \leq 1,
\end{align}
which holds for any $\alpha_1, \alpha_2 \subseteq \{1,\dots,n\}$ and any positive definite $A$. For more information see Theorems 7.8.1, 7.8.5, and 7.8.9 in \cite{Horn:1990}. 
Various generalizations and extensions of inequality~\eqref{2022-07-27-ineq-3} appear in \cite{Barrett:1989}, \cite{Braun:2023},  \cite{Choi:2016}, \cite{Dong:2022}, \cite{Fallat:2001}, \cite{Fallat:2003}, \cite{Fu:2017},\cite{Johnson:1985},  \cite{Johnson:1993}, \cite{Jiang:2019}, and many other works.

Finding other bounded determinantal ratios is a very difficult problem and proving them is even harder. Necessary and sufficient conditions (but not both) were formulated in \cite{Johnson:1993}. In \cite{TracyHall:2008}, the authors conjectured that
\begin{align}
 \label{2026-06-19-R1}
R_1(A) := \frac{A(\{\{1,2,4\}, \{1,3,4\}, \{2,3\}, \{1\}, \{4\} \})}{A(\{ \{1,2\}, \{1,3\}, \{1,4\}, \{2,4\}, \{3,4\} \})} \le \frac{27}{16}.
\end{align}
This upper bound was proven in \cite{sendov2026}, where a sequence of positive semidefinite matrices was given for which the ratio approaches $27/16.$

Two more conjectures were stated in \cite{TracyHall:2008}, namely that the supremums of the following ratios are $1$
\begin{align} 
    R_2(A) &= \frac{A(\{\{1,2,3,4\},\{1,2,3,4\}, \{2,3\},\{2,4\},\{3,4\},\{1\}, \emptyset\})}{A(\{\{1,2,3\}, \{1,2,4\}, \{1,3,4\}, \{2,3,4\}, \{2\},\{3\},\{4\}\})}, \nonumber \\
     R_3(A) &= \frac{A(\{\{1,2,3,4\},\{2,3\}, \{2,4\},\{3,4\},\{1\},\{1\}, \emptyset\})}{A(\{\{2,3,4\},\{1,2\},\{1,3\},\{1,4\},\{2\},\{3\},\{4\}\})}, \label{2026-06-19-Ri}
\end{align}
over all $4 \times 4$ positive definite matrices $A$.

By dividing appropriate rows and columns in the numerators and the denominators by the square root of the positive entries $a_{11}, a_{22}, a_{33}, a_{44}$, one can assume that $A$ is a correlation matrix, that is
\begin{align}
\label{2026-04-02-assum}
    a_{11} = a_{22} = a_{33} = a_{44} = 1.
\end{align}
The positive definiteness implies that $a_{ij} \in (-1, 1)$ for all $1 \le i < j \le 4$. By multiplying rows and corresponding columns in the numerator and the denominator by $-1$, if necessary, we can further assume that
\begin{align}
\label{2026-06-09-assumption}
    a_{12}, a_{13}, a_{14} \in [0, 1).
\end{align}
An important observation is that both ratios $R_2(A)$ and $R_3(A)$ are symmetric with respect to the indices $\{2, 3, 4\}$. That is, the value of the ratio is invariant under any permutation of the set $\{2, 3, 4\}$.

To simplify the presentation, we adopt the notation $\Delta_\alpha := \det A[\alpha]$ for any index set $\alpha \subseteq \{1, 2, 3, 4\}$. For specific index sets, we simply list the indices as subscripts, e.g., $\Delta_{ij} := \det A[\{i,j\}]$. Under assumption \eqref{2026-04-02-assum}, the ratios take the following form:
\begin{equation}
\label{eq:R3_delta}
    R_2(A) = \frac{\Delta_{1234} \Delta_{1234} \Delta_{23} \Delta_{24} \Delta_{34}}{\Delta_{123} \Delta_{124} \Delta_{134} \Delta_{234}} \quad \mbox{and} \quad
    R_3(A) = \frac{\Delta_{1234} \Delta_{23} \Delta_{24} \Delta_{34}}{\Delta_{234} \Delta_{12} \Delta_{13} \Delta_{14}}.
\end{equation}

Our strategy is to look for the supremums of the ratios using optimization techniques. The domain over which we are optimizing is 
\begin{align}
\label{2026-04-09-D}
D:=\{A \in \mathbb{R}^{4 \times 4}: A \succ 0, \mbox{diag\,}(A) = (1,1,1,1)\}.
\end{align}
In the remainder of the paper, we prove the following result.

\begin{theorem} \label{thm:R2}
We have $R_2(A) \le 1$ and $R_3(A) \le 1$ for all $4 \times 4$ positive definite matrices $A$.
\end{theorem}

\begin{proof}
Suppose on the contrary that
\(\sup_D R_i>1\), $i=2,3$. Choose a sequence \(\{A_n\}\subset  D\), such that
\(R_i(A_n)\) converges to \(\sup_D R_i\). Since the closure of \(D\) in the affine space
\(\{A \in \mathbb{R}^{4 \times 4}:  \operatorname{diag}(A)=(1,1,1,1)\}\) is compact, we can assume without loss of generality that $\{A_n\}$ converges to some
positive (semi) definite correlation matrix \(A_0\). 
If $A_0 \in D$, then it has to be a critical point and $R_i(A_0) = 1$ as shown by Theorem~\ref{prop:interior_critical}. 
(Section~\ref{2026-06-18-sect} is devoted to the proof of that theorem.) Otherwise \(A_0\in\partial D\) and then 
$\limsup_{n\to\infty}R_i(A_n)\le1,$ as shown in Section~\ref{2026-06-17-sect}. These contradictions conclude the proof.
\end{proof}

To avoid cumbersome indexing notation, we index the entries of any principal submatrix $A[\alpha]$ and its inverse using the original global indices from the full matrix $A$. Thus, the $(i,j)$ entry of $A[\alpha]^{-1}$ corresponds to the original row $i$ and column $j$ of $A$, for $i, j \in \alpha$. For example, if $\alpha = \{2, 3\}$, the top-right entry of the $2 \times 2$ inverse matrix $A[\alpha]^{-1}$ is denoted as $(A[\alpha]^{-1})_{23}$ instead of the conventional $(A[\alpha]^{-1})_{12}$.

The following result, see \cite[Theorem 6.1]{JuarezRuiz:2016}, is useful in the sequel. 

\begin{theorem}
\label{lem:submatrix_inv}
Let $M$ be a non-singular matrix partitioned as
\begin{equation}
    M = \begin{pmatrix} P & Q \\ R & S \end{pmatrix} \mbox{ and let }     M^{-1} = \begin{pmatrix} \tilde{P} & \tilde{Q} \\ \tilde{R} & \tilde{S} \end{pmatrix}
\end{equation}
be its inverse partitioned conformably.
If the block $\tilde{S}$ is non-singular, then the principal submatrix $P$ is also non-singular, and its inverse is given by the Schur's complement of $\tilde{S}$ in $M^{-1}$, namely
\begin{equation}
    P^{-1} = \tilde{P} - \tilde{Q}\tilde{S}^{-1}\tilde{R}.
\end{equation}
\end{theorem}

\section{Critical points in the relative interior of the domain}
\label{2026-06-18-sect}

The goal of this section is to prove the following theorem.

\begin{theorem} 
\label{prop:interior_critical}
The critical points of the determinantal ratios $R_2(A)$ and $R_3(A)$, in the relative interior of $D$, satisfy  \begin{align*}
a_{14}&=a_{24}=a_{34}=0 \mbox{ and } a_{23}=a_{12}a_{13},  \mbox{ or} \\
a_{12}&=a_{24}=a_{23}=0 \mbox{ and } a_{34}=a_{13}a_{14},  \mbox{ or} \\
a_{13}&=a_{23}=a_{34}=0 \mbox{ and } a_{24}=a_{12}a_{14}. 
\end{align*}
Both ratios evaluate to $1$ on the critical points.
\end{theorem}

\subsection{The critical points of $R_2$}

We start by finding the system of equations for the critical points of $R_2$ in the relative interior of the domain. For that we analyze the partial derivatives of the logarithm of the original ratio $R_2(A)$ 
\begin{align*}
    \ln R_2(A) = 2\ln \det A + \ln \Delta_{23} + \ln \Delta_{24} &+ \ln \Delta_{34}  \\
        & - \ln \Delta_{123} - \ln \Delta_{124} - \ln \Delta_{134} - \ln \Delta_{234}.
\end{align*}

Let $B = A^{-1}$ denote the inverse of the positive definite correlation matrix $A$. Note that 
$$
b_{ii} > 0, \mbox{ for all $i=1,2,3,4$}.
$$
By Jacobi's formula, the derivative of the log-determinant of a positive definite matrix $M$ with respect to an off-diagonal entry $m_{ij}$ is given by 
$$
\frac{\partial \ln \det M}{\partial m_{ij}} = 2M^{-1}_{ij}.
$$
Thus, taking the partial derivative of $\ln R_2(A)$ with respect to $a_{12}$, we observe that the terms $\Delta_{23}, \Delta_{24}, \Delta_{34}, \Delta_{134},$ and $\Delta_{234}$ vanish as they do not depend on $a_{12}$. This yields:
\begin{align}
\label{2026-04-02-dlnR2}
    \frac{\partial \ln R_2(A)}{\partial a_{12}} = 4 b_{12} - 2A[\{1,2,3\}]^{-1}_{12} - 2A[\{1,2,4\}]^{-1}_{12}.
\end{align}
To simplify the inverse submatrix entries, we utilize Theorem~\ref{lem:submatrix_inv}. Consider the partition of $A$ and its inverse $B$ obtained by isolating the fourth row and column:
\begin{equation*}
    A = \begin{pmatrix} A[\{1,2,3\}] & v \\ v^T & 1 \end{pmatrix}, \quad B = \begin{pmatrix} B[\{1,2,3\}] & u \\ u^T & b_{44} \end{pmatrix},
\end{equation*}
where $v = (a_{14}, a_{24}, a_{34})^T$ and $u = (b_{14}, b_{24}, b_{34})^T$. The inverse of the principal submatrix $A[\{1,2,3\}]$ can be expressed as the Schur's complement of $b_{44}$ in $B$:
\begin{align*}
    A[\{1,2,3\}]^{-1} = B[\{1,2,3\}] - \frac{u u^T}{b_{44}}.
\end{align*}
Extracting the $(1,2)$ entry from both sides of this matrix equation yields:
\begin{align*}
    A[\{1,2,3\}]^{-1}_{12} = b_{12} - \frac{b_{14}b_{24}}{b_{44}}.
\end{align*}
By applying identical partition logic to isolate the third row and column, we obtain the analogous expression for the submatrix on indices $\{1,2,4\}$:
\begin{align*}
    A[\{1,2,4\}]^{-1}_{12} = b_{12} - \frac{b_{13}b_{23}}{b_{33}}.
\end{align*}
Substituting these identities in \eqref{2026-04-02-dlnR2} cancels the $b_{12}$ terms, leaving:
\begin{align*} 
    \frac{\partial \ln R_2(A)}{\partial a_{12}} =    2\left(\frac{b_{14}b_{24}}{b_{44}} + \frac{b_{13}b_{23}}{b_{33}}\right).
\end{align*}
By the permutation symmetry of the indices $\{2, 3, 4\}$, differentiating with respect to $a_{13}$ and $a_{14}$ yields two analogous equations. Clearing the positive denominators, we obtain the first three equations for the critical points:
\begin{align}
    b_{14}b_{24}b_{33} + b_{13}b_{23}b_{44} &= 0, \label{eq:R2_sys1} \\
    b_{14}b_{34}b_{22} + b_{12}b_{23}b_{44} &= 0, \label{eq:R2_sys2} \\
    b_{12}b_{24}b_{33} + b_{13}b_{34}b_{22}&= 0. \label{eq:R2_sys3}
\end{align}

Next, we consider the partial derivative with respect to $a_{23}$:
\begin{align}
\label{2026-04-02-dlnR223}
    \frac{\partial \ln R_2(A)}{\partial a_{23}} = 4 b_{23} + 2A[\{2,3\}]^{-1}_{23} - 2A[\{1,2,3\}]^{-1}_{23} - 2A[\{2,3,4\}]^{-1}_{23}.
\end{align}
Applying Theorem~\ref{lem:submatrix_inv} again to the $3 \times 3$ submatrices gives:
\begin{align*}
    A[\{1,2,3\}]^{-1}_{23} = b_{23} - \frac{b_{24}b_{34}}{b_{44}} \quad \mbox{and}  \quad
    A[\{2,3,4\}]^{-1}_{23} = b_{23} - \frac{b_{12}b_{13}}{b_{11}}.
\end{align*}
Substituting these into \eqref{2026-04-02-dlnR223} cancels the $4 b_{23}$ terms. Noting that the off-diagonal entry of the $2 \times 2$ inverse matrix is $A[\{2,3\}]^{-1}_{23} = -a_{23}/(1-a_{23}^2)$, we obtain:
\begin{align*}
  \frac{\partial \ln R_2(A)}{\partial a_{23}} =  2\left(\frac{-a_{23}}{1-a_{23}^2} + \frac{b_{24}b_{34}}{b_{44}} + \frac{b_{12}b_{13}}{b_{11}}\right).
\end{align*}
By symmetry, differentiating with respect to $a_{24}$ and $a_{34}$ provides the remaining two equations, completing the full system of six conditions for the critical points:
\begin{align}
    \frac{a_{23}}{1-a_{23}^2} &= \frac{b_{24}b_{34}}{b_{44}} + \frac{b_{12}b_{13}}{b_{11}}, \label{eq:R2_sys4} \\
    \frac{a_{24}}{1-a_{24}^2} &= \frac{b_{23}b_{34}}{b_{33}} + \frac{b_{12}b_{14}}{b_{11}}, \label{eq:R2_sys5} \\
    \frac{a_{34}}{1-a_{34}^2} &= \frac{b_{23}b_{24}}{b_{22}} + \frac{b_{13}b_{14}}{b_{11}}. \label{eq:R2_sys6}
\end{align}

We now proceed to solving the system \eqref{eq:R2_sys1}--\eqref{eq:R2_sys3}, \eqref{eq:R2_sys4}--\eqref{eq:R2_sys6}.  Rewrite \eqref{eq:R2_sys3} as $b_{12}b_{24}b_{33} = - b_{13}b_{34}b_{22}$. Multiply equation \eqref{eq:R2_sys1} by $b_{12}$ and substitute this relation to get:
\begin{equation}
    b_{14}(- b_{13}b_{34}b_{22}) + b_{13}(b_{12}b_{23}b_{44}) = 0.
\end{equation}
From equation \eqref{eq:R2_sys2}, we know $b_{12}b_{23}b_{44} = - b_{14}b_{34}b_{22}$. Substituting this into the second term gives $-2 b_{13} b_{14} b_{34} b_{22} = 0$. Since $b_{22} > 0$, we obtain $b_{13} b_{14} b_{34} = 0$. The permutation symmetry with respect of the indices $\{2,3,4\}$, yields two additional constraints: $b_{12} b_{14} b_{24} = 0$ and $b_{12} b_{13} b_{23} = 0$.

{\bf Case 1.} Suppose all entries $b_{12}, b_{13}, b_{14}$ are non-zero. This forces $b_{23} = b_{24} = b_{34} = 0$. Consequently, the off-diagonal entries of $B$ are non-zero only along the first row and column. The identity $AB = I_4$ dictates that the off-diagonal entries of the product must vanish. Specifically, for $j \in \{2, 3, 4\}$, the condition $(AB)_{1j} = 0$ together with $a_{11} = 1$ yields 
\begin{align}
\label{2026-04-02-b1j}
b_{1j} = -a_{1j}b_{jj}, \mbox{ for } j \in \{2, 3, 4\}.
\end{align}
Furthermore, for distinct $i, j \in \{2, 3, 4\}$, the condition $(AB)_{ij} = 0$ expands to $a_{i1}b_{1j} + a_{ij}b_{jj} = 0$. Substituting $b_{1j}$ into this equation gives $-a_{1i}a_{1j}b_{jj} + a_{ij}b_{jj} = 0$. Dividing by $b_{jj}$ implies that the entries of $A$ must satisfy the relations 
$$
a_{ij} = a_{1i}a_{1j}, \mbox{ for all $2 \le i < j \le 4$}.
$$

Next, for $j \in \{2, 3, 4\}$, the condition $(AB)_{jj} = 1$ expands to $a_{j1}b_{1j} + a_{jj}b_{jj} = 1$. Substituting $b_{1j} = -a_{1j}b_{jj}$ and $a_{jj} = 1$, we obtain $-a_{1j}^2 b_{jj} + b_{jj} = 1$, which yields:
\begin{equation} \label{eq:b_jj}
    b_{jj} = \frac{1}{1-a_{1j}^2}, \quad \text{for } j \in \{2, 3, 4\}.
\end{equation}
Similarly, evaluating $(AB)_{11} = 1$ gives $b_{11} + a_{12}b_{12} + a_{13}b_{13} + a_{14}b_{14} = 1$. Substituting $b_{1j}$ and the expressions for $b_{jj}$ from \eqref{eq:b_jj}, we derive a representation for $b_{11}$:
\begin{equation} \label{eq:b11_sum}
    b_{11} = 1 + \frac{a_{12}^2}{1-a_{12}^2} + \frac{a_{13}^2}{1-a_{13}^2} + \frac{a_{14}^2}{1-a_{14}^2}.
\end{equation}

Now, we substitute the established relations into the partial derivative equations. Consider equation \eqref{eq:R2_sys4}. Because $b_{24} = b_{34} = 0$, the first term on the right-hand side vanishes. Using \eqref{2026-04-02-b1j} and $a_{23} = a_{12}a_{13}$, transforms the equation into:
\begin{equation} \label{eq:R2_sys4_sub}
    \frac{a_{12}a_{13}}{1-a_{12}^2a_{13}^2} = \frac{(-a_{12}b_{22})(-a_{13}b_{33})}{b_{11}} = \frac{a_{12}a_{13}b_{22}b_{33}}{b_{11}}.
\end{equation}
Under the assumption in this case $b_{1j} \neq 0$, so the elements $a_{12}$ and $a_{13}$ are non-zero. Canceling $a_{12}a_{13}$ from both sides, solving for $b_{11}$, and substituting $b_{22}$ and $b_{33}$ from \eqref{eq:b_jj}, gives a second representation for $b_{11}$:
\begin{equation} \label{eq:b11_prod}
    b_{11} = \frac{1-a_{12}^2a_{13}^2}{(1-a_{12}^2)(1-a_{13}^2)} = 1 + \frac{a_{12}^2}{1-a_{12}^2} + \frac{a_{13}^2}{1-a_{13}^2}.
\end{equation}
Equating the two expressions for $b_{11}$ from \eqref{eq:b11_sum} and \eqref{eq:b11_prod} forces $a_{14}^2/(1-a_{14}^2) = 0$, which dictates $a_{14} = 0$. By \eqref{2026-04-02-b1j}, we have $b_{14}=0$, which contradicts the assumption in this case, showing that this case is not possible.

{\bf Case 2.} Suppose at least one of the entries $b_{12}, b_{13}, b_{14}$ is zero.
Without loss of generality, due to index symmetry, assume $b_{14} = 0$. Substituting $b_{14} = 0$  into  equations \eqref{eq:R2_sys1}--\eqref{eq:R2_sys3}, reduces them to:
\begin{align}
    b_{13}b_{23}b_{44} &= 0, \nonumber \\
    b_{12}b_{23}b_{44} &= 0, \nonumber \\
    b_{12}b_{24}b_{33} + b_{13}b_{34}b_{22} &= 0. \label{2026-04-03-eqn}
\end{align}
Since the diagonal entries $b_{44}$ and $b_{22}$ are strictly positive, the first two equations imply that either $b_{23} = 0$, or $b_{12} = b_{13} = 0$.

{\bf Case 2a.} If $b_{23} = 0$, we substitute $b_{14}=0$ and $b_{23}=0$ into equations \eqref{eq:R2_sys5} and \eqref{eq:R2_sys6}. The right-hand sides vanish completely, yielding $a_{24} = 0$ and $a_{34} = 0$. With $a_{24} = a_{34} = 0$, the entry $b_{14}$ of the inverse matrix evaluates strictly to $-a_{14}(1-a_{23}^2) / \det A$. Since $A$ is positive definite, the principal minor satisfies $(1-a_{23}^2) > 0$. Thus, the constraint $b_{14} = 0$ implies $a_{14} = 0$. In a similar way, the constraint $b_{23} = 0$ implies the condition $a_{23} = a_{12}a_{13}$. Let us see now that equation \eqref{eq:R2_sys4} holds. Indeed, the conditions $a_{14}=a_{24}=a_{34}=0$ imply that $b_{24}=b_{34}=0$ and $b_{44}=1$. This verifies equation \eqref{2026-04-03-eqn}. With $a_{23} = a_{12}a_{13}$, the remaining entries of the inverse matrix $B$ are:
\begin{equation}
    b_{11} = \frac{1-a_{12}^2a_{13}^2}{(1-a_{12}^2)(1-a_{13}^2)}, \quad b_{12} = \frac{-a_{12}}{1-a_{12}^2}, \quad b_{13} = \frac{-a_{13}}{1-a_{13}^2}.
\end{equation}
Substituting these into the right-hand side of \eqref{eq:R2_sys4} verifies its validity. We are free to choose any values of  $a_{12}$ and $a_{13}$ that maintain positive definiteness. Thus, this case gives the first solution in Theorem~\ref{prop:interior_critical}. The rest are derived by the symmetry.

{\bf Case 2b.} If $b_{12} = b_{13} = 0$, then together with $b_{14} = 0$, the first row and column of the inverse matrix $B$ contain zeros off the diagonal. This implies that $a_{12} = a_{13} = a_{14} = 0$.
Equations \eqref{eq:R2_sys1}--\eqref{eq:R2_sys3} are automatically satisfied. Under these conditions, the entries of $B$ in terms of those of $A$ are $b_{11}=1$ and
\begin{align*}
b_{22} &= \frac{1-a_{34}^2}{\det A}, b_{23}= -\frac{a_{23}-a_{24}a_{34}}{\det A}, b_{24}= -\frac{a_{24}-a_{23}a_{34}}{\det A}, \\
b_{33} &=  \frac{1-a_{24}^2}{\det A}, b_{34} = -\frac{a_{34}-a_{23}a_{24}}{\det A}, 
b_{44}= \frac{1-a_{23}^2}{\det A}.
\end{align*}
Substituting into equations \eqref{eq:R2_sys4}--\eqref{eq:R2_sys6} and simplifying gives 
$$
a_{24}a_{34} = a_{23}, \,\,\, a_{23}a_{34} = a_{24}, \,\,\, a_{23}a_{24} = a_{34}. 
$$
The only solution of this system that gives a positive definite $A$ is $a_{23}=a_{24}=a_{34}=0.$
This is a special case of the first solution in Theorem~\ref{prop:interior_critical}.

It is straightforward to verify that when $a_{14}=a_{24}=a_{34}=0$ and $a_{23} = a_{12}a_{13}$, we have $R_2(A)=1$. Consequently, to establish the global inequality $R_2(A) \le 1$, it remains to verify that the ratio does not strictly exceed $1$ as the matrix approaches the boundary of the domain.

\subsection{The critical points of $R_3$} 

We begin by characterizing the critical points of $R_3$ in the relative interior of its domain.  We look at the partial derivatives of the logarithm of $R_3(A)$:
\begin{equation*}
    \ln R_3(A) = \ln \det A + \ln \Delta_{23} + \ln \Delta_{24} + \ln \Delta_{34} - \ln \Delta_{234} - \ln \Delta_{12} - \ln \Delta_{13} - \ln \Delta_{14}.
\end{equation*}
Taking the partial derivative with respect to $a_{12}$, we observe that the terms involving $\Delta_{23}, \Delta_{24}, \Delta_{34}, \Delta_{234}, \Delta_{13},$ and $\Delta_{14}$ vanish. This yields:
\begin{equation*}
    \frac{\partial \ln R_3(A)}{\partial a_{12}} = 2 b_{12} - 2A[\{1,2\}]^{-1}_{12}.
\end{equation*}
For the $2 \times 2$ principal submatrix $A[\{1,2\}]$, its inverse can be explicitly computed. Extracting the off-diagonal entry gives $A[\{1,2\}]^{-1}_{12} = -a_{12}/(1-a_{12}^2)$. Substituting this into the derivative equation yields:
\begin{equation} \label{eq:R3_sys1}
    b_{12} + \frac{a_{12}}{1-a_{12}^2} = 0.
\end{equation}
By the permutation symmetry of the indices $\{2, 3, 4\}$, differentiating with respect to $a_{13}$ and $a_{14}$ yields two analogous equations:
\begin{align}
    b_{13} + \frac{a_{13}}{1-a_{13}^2} &= 0, \label{eq:R3_sys2} \\
    b_{14} + \frac{a_{14}}{1-a_{14}^2} &= 0. \label{eq:R3_sys3}
\end{align}

Next, we consider the partial derivative with respect to $a_{23}$:
\begin{equation*}
    \frac{\partial \ln R_3(A)}{\partial a_{23}} = 2 b_{23} + 2A[\{2,3\}]^{-1}_{23} - 2A[\{2,3,4\}]^{-1}_{23}.
\end{equation*}

To simplify the inverse submatrix entries, consider the partition of $A$ and its inverse $B = A^{-1}$ isolating the first row and column:
\begin{equation*}
    A = \begin{pmatrix} 1 & v^T \\ v & A[\{2,3,4\}] \end{pmatrix}, \quad B = \begin{pmatrix} b_{11} & u^T \\ u & B[\{2,3,4\}] \end{pmatrix},
\end{equation*}
where $v = (a_{12}, a_{13}, a_{14})^T$ and $u = (b_{12}, b_{13}, b_{14})^T$. By Theorem~\ref{lem:submatrix_inv}, the inverse of the principal submatrix $A[\{2,3,4\}]$ is the Schur's complement of $b_{11}$ in $B$:
\begin{equation*}
    A[\{2,3,4\}]^{-1} = B[\{2,3,4\}] - \frac{u u^T}{b_{11}}.
\end{equation*}
Extracting the entry corresponding to the indices $\{2,3\}$ from both sides of this matrix equation gives:
\begin{equation*}
    A[\{2,3,4\}]^{-1}_{23} = b_{23} - \frac{b_{12}b_{13}}{b_{11}}.
\end{equation*}
Additionally, the off-diagonal entry of the $2 \times 2$ inverse matrix is $A[\{2,3\}]^{-1}_{23} = -a_{23}/(1-a_{23}^2)$. Substituting these two identities into the derivative cancels the $2 b_{23}$ terms and after easy simplification, we obtain
\begin{equation} \label{eq:R3_sys4}
    \frac{b_{12}b_{13}}{b_{11}} - \frac{a_{23}}{1-a_{23}^2} = 0.
\end{equation}
By symmetry, differentiating with respect to $a_{24}$ and $a_{34}$ provides the remaining two equations, completing the full system of six equations:
\begin{align}
    \frac{b_{12}b_{14}}{b_{11}} - \frac{a_{24}}{1-a_{24}^2} &= 0, \label{eq:R3_sys5} \\
    \frac{b_{13}b_{14}}{b_{11}} - \frac{a_{34}}{1-a_{34}^2} &= 0. \label{eq:R3_sys6}
\end{align}

To solve this system, we look at the matrix identity $AB = I_4$. For brevity, we introduce the auxiliary variables $x_j = {a_{1j}}/({1-a_{1j}^2})$ for $j \in \{2, 3, 4\}$. From the first subsystem \eqref{eq:R3_sys1}--\eqref{eq:R3_sys3}, the off-diagonal entries in the first row of $B$ can be expressed as $b_{1j} = -x_j$. 

Evaluating the $(1,1)$ entry of the product $AB = I_4$ yields 
$$
a_{11}b_{11} + a_{12}b_{12} + a_{13}b_{13} + a_{14}b_{14} = 1.
$$ 
Substituting $a_{11} = 1$ and $b_{1j} = -x_j$ provides an explicit representation for $b_{11}$:
\begin{equation} \label{eq:b11_explicit}
    b_{11} = 1 + a_{12}x_2 + a_{13}x_3 + a_{14}x_4.
\end{equation}
Next, the $(2,1)$ entry of $AB = I_4$ is 
$$
a_{12}b_{11} + a_{22}b_{12} + a_{23}b_{13} + a_{24}b_{14} = 0.
$$
Substitute $a_{22}=1$ and $b_{j1} = -x_j$ to obtain:
\begin{equation*}
    a_{12}b_{11} - x_2 - a_{23}x_3 - a_{24}x_4 = 0.
\end{equation*}
Substituting the explicit formula for $b_{11}$ from \eqref{eq:b11_explicit} into this equation gives:
\begin{align*}
    a_{12}&(1 + a_{12}x_2 + a_{13}x_3 + a_{14}x_4) - x_2 - a_{23}x_3 - a_{24}x_4 \\
    &=  a_{12} - (1-a_{12}^2) x_2 + (a_{12}a_{13}-a_{23})x_3 + (a_{12}a_{14}-a_{24})x_4 \\
    &= (a_{12}a_{13}-a_{23})x_3 + (a_{12}a_{14}-a_{24})x_4 = 0,
\end{align*}
where we used the definition of $x_2$.
By the permutation symmetry of the indices $\{2, 3, 4\}$, evaluating the $(3,1)$ and $(4,1)$ entries of $AB = I_4$ generates two analogous equations. Defining $U := a_{12}a_{13}-a_{23}$, $V := a_{12}a_{14}-a_{24}$, and $W := a_{13}a_{14}-a_{34}$, we obtain a homogeneous linear system in terms of the `variables' $(U, V, W)$:
\begin{align}
\label{eq:homo3}
\begin{array}{rcrcrl}
    U x_3 & +  & V x_4  &      &            & = 0, \\
    U x_2  &    &             &  + & W x_4 &= 0, \\
               &    &  V x_2  & + & W x_3  &= 0. 
\end{array}
\end{align}
The determinant of this system is $-2 x_2 x_3 x_4$. 

{\bf Case 1.} Suppose the entries $a_{12}, a_{13}, a_{14}$ are all non-zero. Consequently, $x_2, x_3, x_4$ are non-zero, making the determinant strictly non-zero. This forces the homogeneous system to have only the trivial solution: $U = V = W = 0$. In particular, $U = 0$ implies $a_{23} = a_{12}a_{13}$.

However, substituting $a_{23} = a_{12}a_{13}$ back into the partial derivative equation \eqref{eq:R3_sys4} yields:
\begin{equation*}
 \frac{a_{12} a_{13}}{(1-a_{12}^2)(1-a_{13}^2) b_{11}} -   \frac{a_{12}a_{13}}{1-a_{12}^2 a_{13}^2} = 0.
\end{equation*}
Solving this for $b_{11}$ gives:
\begin{equation} \label{eq:b11_contradiction}
    b_{11} = \frac{1-a_{12}^2 a_{13}^2}{(1-a_{12}^2)(1-a_{13}^2)} = 1 + \frac{a_{12}^2}{1-a_{12}^2} + \frac{a_{13}^2}{1-a_{13}^2}.
\end{equation}
Comparing this with \eqref{eq:b11_explicit}, one sees that the equality holds if and only if $a_{14}x_4 = {a_{14}^2}/({1-a_{14}^2}) = 0$, which mandates $a_{14} = 0$. This contradiction shows that this case is not possible.

{\bf Case 2.} Suppose at least one of the  entries $a_{12}, a_{13}, a_{14}$ is zero.
Without loss of generality, due to the symmetry, assume $a_{14} = 0$, which implies $x_4 = 0$ and $b_{14} = 0$. 

Substituting $b_{14} = 0$ into the partial derivative equations \eqref{eq:R3_sys5} and \eqref{eq:R3_sys6} yields ${a_{24}}/({1-a_{24}^2}) = 0$ and ${a_{34}}/({1-a_{34}^2}) = 0$ implying that $a_{24} = 0$ and $a_{34} = 0$.  Next, substituting $x_4 = 0$ into the homogeneous system \eqref{eq:homo3} reduces the first equation to $U x_3 = 0$. 

{\bf Case 2a.} If both $a_{12}$ and $a_{13}$ are not zero, then $x_3 \neq 0$, which forces $U = 0$, that is $a_{23} = a_{12}a_{13}$.

{\bf Case 2b.} If $a_{12} = 0$, then $x_2 = 0$ and consequently $b_{12} = 0$. Substituting $b_{12} = 0$ into equation \eqref{eq:R3_sys4} eliminates the first term, leaving $-{a_{23}}/({1-a_{23}^2}) = 0$, which forces $a_{23} = 0$. The relationship $a_{23} = a_{12}a_{13}$ remains valid.

If $a_{13} = 0$, then we analogously conclude that $a_{23} = 0$ and the relationship $a_{23} = a_{12}a_{13}$ remains valid. 

This gives the first solution in Theorem~\ref{prop:interior_critical}. The rest are derived by the symmetry.

It is straightforward to verify that when $a_{14}=a_{24}=a_{34}=0$ and $a_{23} = a_{12}a_{13}$, we have $R_3(A)=1$. Consequently, to establish the global inequality $R_3(A) \le 1$, it remains to verify that the ratio does not  exceed $1$ as the matrix approaches the boundary of the domain.

\section{Boundary Analysis}
\label{2026-06-17-sect}

To prove the global inequalities $R_2(A) \le 1$ and $R_3(A) \le 1$, we need to analyze the behavior of the ratios as the correlation matrix $A$ approaches the boundary of the positive definite cone. Throughout this section, we assume that 
$$
A \to A_0 \in \partial D, 
$$
where $D$ is defined by \eqref{2026-04-09-D}. Since $A$ satisfies assumption \eqref{2026-06-09-assumption}, the limit $A_0$ satisfies it too.  Let $\lambda_1 \ge \cdots \ge \lambda_4 > 0$ be the eigenvalues of $A$ with corresponding orthonormal eigenvectors $u_1,\ldots, u_4$.  As $A \to A_0$, the eigenvalues $\lambda_i$ converge to the eigenvalues $\mu_i$ of $A_0$, and the eigenvectors $u_i :=(u_{1i}, u_{2i}, u_{3i}, u_{4i})^T$ can be chosen to converge to corresponding orthonormal eigenvectors $v_i :=(v_{1i}, v_{2i}, v_{3i}, v_{4i})^T$ of $A_0$, $i=1,\ldots, 4$.

The inverse of $A$ is given by 
$$
A^{-1} = \frac{1}{\lambda_4} u_4 u_4^T + \sum_{i=1}^3 \frac{1}{\lambda_i} u_i u_i^T =: \frac{1}{\lambda_4} u_4 u_4^T + M.
$$
By the identity $\text{adj}(A) = (\det A)A^{-1}$, we have $\text{adj}(A) = \lambda_1 \lambda_2 \lambda_3(u_4u_4^T + \lambda_4 M)$. Evaluating the diagonal entries of this matrix, which correspond to the $3 \times 3$ principal minors of $A$, gives:
\begin{align}
    \Delta_{234}(A) &= \lambda_1 \lambda_2 \lambda_3 (u_{14}^2 + \lambda_4 M_{11}), \quad \Delta_{134}(A) = \lambda_1 \lambda_2 \lambda_3 (u_{24}^2 + \lambda_4 M_{22}), \nonumber \\
    \Delta_{124}(A) &= \lambda_1 \lambda_2 \lambda_3 (u_{34}^2 + \lambda_4 M_{33}), \quad \Delta_{123}(A) = \lambda_1 \lambda_2 \lambda_3 (u_{44}^2 + \lambda_4 M_{44}). \label{eq:minor_spectral}
\end{align}

\subsection{Suppose $A_0$ has rank $3$} 
In this case, the eigenvalues of  $A_0$ are $\mu_1\ge \mu_2 \ge \mu_3 > \mu_4 = 0$ with $A_0 v_4 = 0$ and $\|v_4\| = 1$. 

\subsubsection{Boundary behaviour of $R_2$} 

Substituting \eqref{eq:minor_spectral} into $R_2(A)$ provides a path-independent representation of the ratio: 
\begin{align}
\label{2026-06-16-R2}
R_2(A) = \frac{\lambda_4^2 \Delta_{23}(A) \Delta_{24}(A) \Delta_{34}(A)}{(\lambda_1 \lambda_2 \lambda_3)^2 \prod_{i=1}^4 (u_{i4}^2 + \lambda_4 M_{ii})}.
\end{align}
We investigate the limit  of this ratio based on the sparsity of the null vector $v_4$. Note that $v_4$ cannot have three zero components. Indeed, if $v_{i4} \not=0$, then the $i$-th
row of $A_0 v_4 = 0$ evaluates to $v_{i4} = 0$, a contradiction.

{\bf Case 1.} \textit{Vector $v_4$ has no zero components.} In this case, the denominator of $R_2(A)$ approaches $(\mu_1 \mu_2 \mu_3)^2 v_{14}^2 \cdots v_{44}^2>0$. Since $\lambda_4^2 \to 0$, we conclude that $\lim_{A \to A_0} R_2(A) = 0$.

{\bf Case 2.} \textit{The vector \(v_4\) has exactly one zero component.}
Suppose \(v_{k4}=0\) and \(v_{i4}\ne0\) for \(i\ne k\). By definition,
\[
M_{kk}
=
\sum_{r=1}^3\frac{u_{kr}^2}{\lambda_r}
\to
\sum_{r=1}^3\frac{v_{kr}^2}{\mu_r} =: 2c > 0,
\]
since \(v_1,v_2,v_3,v_4\) form an orthonormal basis and \(v_{k4}=0\) implies
$\sum_{r=1}^3v_{kr}^2=1.$ Thus, for \(A\) sufficiently close to \(A_0\), we have
\[
u_{k4}^2+\lambda_4M_{kk}\ge c\lambda_4.
\]
The other factors \(u_{i4}^2+\lambda_4M_{ii}\) converge to
\(v_{i4}^2>0\), \(i\ne k\). Hence, by taking $c > 0$ smaller, if necessary, we have that for \(A\) sufficiently close to \(A_0\):
$$
(\lambda_1\lambda_2\lambda_3)^2
\prod_{i=1}^4(u_{i4}^2+\lambda_4M_{ii}) 
\ge c\lambda_4.
$$
Since every \(2\times2\) principal minor of a correlation matrix is at most \(1\),
using \eqref{2026-06-16-R2}, we get
\[
0\le R_2(A)
\le
\frac{\lambda_4^2}{c\lambda_4}
=
\frac{\lambda_4}{c}
\to0.
\]

{\bf Case 3a.} \textit{Vector $v_4$ has exactly two zero components, both from the index set $\{2, 3, 4\}$.}
Without loss of generality, assume $v_{34} = v_{44} = 0$ and $v_{14}, v_{24} \neq 0$. The condition $A_0 v_4 = 0$ applied to the first two rows yields $v_{14} + (A_{0})_{12}v_{24} = 0$ and $(A_{0})_{12}v_{14} + v_{24} = 0$. Since 
$(A_{0})_{12} \ge 0$, the system implies that $(A_{0})_{12} = 1$ and $v_{14} = -v_{24}$. Evaluating the third and fourth rows with $v_{14} = -v_{24}$ yields $(A_{0})_{13} = (A_{0})_{23}$ and $(A_{0})_{14} = (A_{0})_{24}$.

By Lemma~\ref{2026-06-12-lem:common_K_bound}, for the positive definite \(A\) we have
\[
R_2(A)\le 
\frac{\Delta_{23}(A)\Delta_{34}(A)}{\Delta_{13}(A)}.
\]

Since \((A_0)_{13}=(A_0)_{23}\), we have
\(\Delta_{13}(A_0)=\Delta_{23}(A_0)\). If \(\Delta_{13}(A_0)>0\), then 
\[
\limsup_{A\to A_0}R_2(A)
\le
\frac{\Delta_{23}(A_0)\Delta_{34}(A_0)}{\Delta_{13}(A_0)}
=
\Delta_{34}(A_0)
\le 1.
\]
If $\Delta_{13}(A_0) = 0$, then $(A_0)_{13} = (A_0)_{23} = 1$. Since $A_0$ is a correlation matrix, its Gram representation, shows that columns $1$, $2$, and $3$ are identical. This contradicts the assumption that the rank of $A_0$ is $3$. 


{\bf Case 3b.} \textit{Vector $v_4$ has exactly two zero components, one of which has index $1$.}
Due to the permutation symmetry, assume $v_{14} = v_{24} = 0$ and $v_{34}, v_{44} \neq 0$. The condition $A_0 v_4 = 0$ 
 applied to the last two rows yields $v_{34} + (A_{0})_{34}v_{44} = 0$ and $(A_{0})_{34}v_{34} + v_{44} = 0$. This implies that $(A_{0})_{34} = \pm 1$ (hence $\Delta_{34}(A_0) = 0$) and $v_{34} = \mp v_{44} $. 

Evaluating the first two rows with $v_{34} = \mp v_{44} $ yields $(A_{0})_{13} = \pm (A_{0})_{14}$ and $(A_{0})_{23} = \pm (A_{0})_{24}$.  If \((A_0)_{13}=(A_0)_{14} = 0\), then $\Delta_{13}(A_0) = 1$, and taking limit superior in the inequality in Lemma~\ref{2026-06-12-lem:common_K_bound}, one sees that the right-hand side is bounded above by one. Thus we can suppose that $(A_{0})_{13} = (A_{0})_{14} > 0$ and one should choose the top sign everywhere above. That is,  \((A_0)_{34}=1\), \((A_0)_{23}=(A_0)_{24}\), and \(v_{34}=-v_{44}\). 

If
\(\Delta_{13}(A_0)>0\), then taking limit superior in the inequality in
Lemma~\ref{2026-06-12-lem:common_K_bound}, one sees that the right-hand side converges to zero, 
since \(\Delta_{34}(A_0)=0\). 

If \(\Delta_{13}(A_0)=0\), then \((A_0)_{13}=1\). Since \(A_0\) is a correlation
matrix, its Gram representation, together with \((A_0)_{34}=1\), shows that
\(x_1,x_3,x_4\) are identical. This contradicts the assumption that the rank of \(A_0\) is $3$.


\subsubsection{Boundary behaviour of $R_3$} 

Using \eqref{eq:minor_spectral}, the ratio $R_3(A)$ can be expressed path-independently as follows
$$
R_3(A) = \frac{\lambda_4 \Delta_{23}(A) \Delta_{24}(A) \Delta_{34}(A)}{(u_{14}^2 + \lambda_4 M_{11}) \Delta_{12}(A) \Delta_{13}(A) \Delta_{14}(A)}.
$$
We consider two cases based on the first component of the null vector $v_4$.

{\bf Case 1.} \textit{Suppose $v_{14} \neq 0$.}
The term $u_{14}^2 + \lambda_4 M_{11}$ converges to $v_{14}^2 > 0$. If the principal minors of $A_0$ satisfy $\Delta_{1j}(A_0) > 0$ for all $j \in \{2, 3, 4\}$, the denominator converges to a positive constant. Since $\lambda_4 \to 0$, we have $R_3(A) \to 0$.

If any of these minors vanish, then by symmetry of the indexes $\{2,3,4\}$, we can assume $\Delta_{12}(A_0) = 0$. Then $(A_0)_{12} = 1$ implies $x_1 = x_2$, and therefore $\Delta_{13}(A_0) = \Delta_{23}(A_0)$ and $\Delta_{14}(A_0) = \Delta_{24}(A_0)$. If \(\Delta_{13}(A_0)>0\), then by Lemma~\ref{2026-06-12-lem:common_K_bound}
\[
\limsup_{A\to A_0}R_3(A)
\le
\limsup_{A\to A_0} \frac{\Delta_{23}(A)\Delta_{34}(A)}{\Delta_{13}(A)}
=
\frac{\Delta_{23}(A_0)\Delta_{34}(A_0)}{\Delta_{13}(A_0)}
=
\Delta_{34}(A_0)
\le1.
\]
If \(\Delta_{13}(A_0)=0\), then \(x_1=x_2=x_3\), so the first, second, and third
columns of \(A_0\) are identical, contradicting the assumption that the rank of \(A_0\) is $3$. 

{\bf Case 2.} \textit{Suppose $v_{14} = 0$.}
The condition $A_0 v_4 = 0$ implies that the principal submatrix $A_0[\{2,3,4\}]$ has a non-trivial null vector $(v_{24}, v_{34}, v_{44})^T$, yielding $\Delta_{234}(A_0) = 0$. Geometrically, this means the Gram vectors $x_2, x_3, x_4$ associated with $A_0$ span a subspace of dimension at most $2$. Because the rank of $A_0$ is $3$, the vector $x_1$ cannot lie in this subspace. Decompose $x_1$ as $x_1=\ell y + x_1^\perp$, where $\ell y$ is the orthogonal projection of $x_1$ onto the span of $x_2, x_3, x_4$ and $x_1^\perp:=x_1- \ell y$ is orthogonal to that span.
We assume that $\|y\|=1$ and $\ell \ge 0$. Since $\|x_1\|=1$, we have $|\ell| \le 1$.

If \(\ell=0\), then \(x_1\perp x_j\) for \(j=2,3,4\), so $\Delta_{12}(A_0)=\Delta_{13}(A_0)=\Delta_{14}(A_0)=1.$
Taking limit superior from both sides of the inequality in 
Lemma~\ref{2026-06-12-lem:common_K_bound}, one sees that the right-hand side is bounded above by one.

Assume now \(\ell  > 0\). By the sign assumption on the entries of $A_0$, we have \(x_1^T x_j\ge0\), hence  \(y^T x_j\ge0\), and
\begin{align}
\label{2026-06-17-est}
\Delta_{1j}(A_0)=1- \ell^2 (y^T x_j)^2\ge 1-(y^T x_j)^2=:\Delta_{yj} \mbox{ for \(j=2,3,4\).} 
\end{align}
The three numbers \(\Delta_{y2},\Delta_{y3},\Delta_{y4}\) cannot all be zero, otherwise
\(x_2,x_3,x_4\) would all be parallel to \(y\), so their span would be one-dimensional,
and \(A_0\) would have rank at most \(2\), a contradiction.  Apply Lemma~\ref{lem:low_rank_bound_sign_free} to the 
correlation matrix, $A_y$, generated by \(y,x_2,x_3,x_4\) and having rank at most \(2\), to obtain that 
\begin{align}
\label{2026-06-17-ineq}
\min \Big( \frac{\Delta_{23}(A_y)\Delta_{34}(A_y)}{\Delta_{y3}},
\frac{\Delta_{23}(A_y)\Delta_{24}(A_y)}{\Delta_{y2}},
\frac{\Delta_{24}(A_y)\Delta_{34}(A_y)}{\Delta_{y4}} \Big) \le 1,
\end{align}
where the minimum is taken over the well-defined ratios.

Suppose, without loss of generality, that the first ratio is well-defined and attains the minimum. Thus,  \(\Delta_{y3}>0\) and 
$\Delta_{13}(A_0)\ge\Delta_{y3}>0$. By
Lemma~\ref{2026-06-12-lem:common_K_bound},
\begin{align*}
\limsup_{A\to A_0} R_3(A)
&\le 
\limsup_{A\to A_0} \frac{\Delta_{23}(A)\Delta_{34}(A)}{\Delta_{13}(A)}
=
\frac{\Delta_{23}(A_0)\Delta_{34}(A_0)}{\Delta_{13}(A_0)} \\
&\le
\frac{\Delta_{23}(A_y)\Delta_{34}(A_y)}{\Delta_{y3}}
\le 1,
\end{align*}
where we used the fact that $\Delta_{ij}(A_0)=\Delta_{ij}(A_y)$ for $ij=23,24,34$.

\subsection{Suppose $A_0$ has rank at most $2$}  

If the minors $\Delta_{12}(A_0)$, $\Delta_{13}(A_0)$, $\Delta_{14}(A_0)$ are not all zero, then without loss of generality assume that $\Delta_{12}(A_0)>0$ and the minimum in Lemma~\ref{lem:low_rank_bound_sign_free} is attained by its first ratio. Taking limit superior from both sides of the inequality in Lemma~\ref{2026-06-12-lem:common_K_bound}, shows that $\limsup_{A \to A_0} R_i(A) \le 1$ for $i=2,3$.

If $\Delta_{12}(A_0)=\Delta_{13}(A_0)=\Delta_{14}(A_0)=0$, then $(A_0)_{12}=(A_0)_{13}=(A_0)_{14}=1$ implying that $x_1=x_2=x_3=x_4$ and so $A_0$ is the all-one matrix. Let \(A_n = (a_{ij}^{n})\) be a sequence of positive definite correlation matrices approaching $A_0$. Write
\[
\delta_{ij}^{n}:=1-a_{ij}^{n}, \,\,\, 
\varepsilon_n:=\max_{i<j}\delta_{ij}^{n} \mbox{ and } m_n:=\max_{j=2,3,4}\delta_{1j}^{n}.
\]
For \(n\) large, we have \(a_{ij}^{n} \in (0,1)\) and \(\varepsilon_n\to0\).
Let \(x_1^{n},\ldots,x_4^{n}\) be unit vectors whose Gram matrix is \(A_n\).
Then
\[
\|x_i^{n}-x_j^{n}\|^2
=
2(1-a_{ij}^{n})
=
2\delta_{ij}^{n}.
\]

We claim that \(m_n\ge \varepsilon_n/4\). Indeed, if the maximum
\(\varepsilon_n\) is attained by a pair of indexes involving index \(1\), this is immediate.
Otherwise, suppose it is attained by a pair \(i,j\in\{2,3,4\}\). Then by the triangle
inequality, we have
\[
\sqrt{2\varepsilon_n}
=
\|x_i^{n}-x_j^{n}\|
\le
\|x_i^{n}-x_1^{n}\|
+
\|x_1^{n}-x_j^{n}\|
\le
2\sqrt{2m_n},
\]
hence \(m_n\ge \varepsilon_n/4\).

For each $n$, choose \(j_n\in\{2,3,4\}\), such that \(\delta_{1j_n}^{n}=m_n\). On the one hand,  we have
\[
\Delta_{1j_n}(A_n)
=
1-(a_{1j_n}^{n})^2
=
(1-a_{1j_n}^{n})(1+a_{1j_n}^{n})
\ge \delta_{1j_n}^{n} = m_n\ge \varepsilon_n/4 ,
\]
On the other hand, for every \(k<\ell\), we have
\[
\Delta_{k\ell}(A_n)
=
(1-a_{k\ell}^{n})(1+a_{k\ell}^{n})
\le 2 \delta_{k\ell}^{n}
\le
2\varepsilon_n.
\]
Using Lemma~\ref{2026-06-12-lem:common_K_bound}, choose the term whose
denominator is \(\Delta_{1j_n}(A_n)\). Its numerator is a product of two among
\(\Delta_{23},\Delta_{24},\Delta_{34}\). Therefore, for $i=2,3$, we have
\[
R_i(A_n)
\le
\frac{(2\varepsilon_n)^2}{\varepsilon_n/4}
=
16\varepsilon_n
\to0.
\]

\section{Apendix}

\begin{lemma}
\label{2026-06-12-lem:common_K_bound}
For every positive definite correlation matrix \(A\), we have
\[
R_i(A)\le \min \Big( \frac{\Delta_{23}\Delta_{24}}{\Delta_{12}},
\frac{\Delta_{23}\Delta_{34}}{\Delta_{13}},
\frac{\Delta_{24}\Delta_{34}}{\Delta_{14}}\Big)
\]
for \(i=2,3\).
\end{lemma}
\begin{proof}
Koteljanskii's inequality gives
$\Delta_{1234}\Delta_{24}\le \Delta_{124}\Delta_{234}$,
and $\Delta_{1234}\Delta_{13}\le \Delta_{123}\Delta_{134}$.
Multiplying these inequalities and substituting into the definition of \(R_2\)
yields
\[
R_2(A)\le \frac{\Delta_{23}\Delta_{34}}{\Delta_{13}}.
\]
The other two bounds follow by the symmetry in the indices \(2,3,4\).

Again by Koteljanskii 
$\Delta_{1234}\Delta_{23}\le \Delta_{123}\Delta_{234}$.
Substituting into $R_3$, gives
\[
R_3(A)\le
\frac{\Delta_{123}\Delta_{24}\Delta_{34}}
{\Delta_{12}\Delta_{13}\Delta_{14}} \le \frac{\Delta_{24}\Delta_{34}}{\Delta_{14}},
\]
where the last inequality holds since
$A$ is a correlation matrix and by Koteljanskii
$\Delta_{123}\le \Delta_{12}\Delta_{13}$. Again, the remaining two bounds follow from the symmetry in \(2,3,4\).
\end{proof}

When the correlation matrix has low rank, the minimum in Lemma~\ref{2026-06-12-lem:common_K_bound} can be estimated directly.

\begin{lemma}\label{lem:low_rank_bound_sign_free}
Let \(A_0\) be a \(4\times4\) positive semidefinite correlation matrix
with rank at most \(2\). Then either
\[
\Delta_{12}=\Delta_{13}=\Delta_{14}=0,
\]
or at least one of the following ratios is well-defined and
\[
\min \left(
\frac{\Delta_{23}\Delta_{24}}{\Delta_{12}},
\frac{\Delta_{23}\Delta_{34}}{\Delta_{13}},
\frac{\Delta_{24}\Delta_{34}}{\Delta_{14}}
\right)\le 1,
\]
where the minimum is taken over the well-defined ratios.
\end{lemma}

\begin{proof}
Let \(A_0\) be the Gram matrix of unit vectors
\(x_1,x_2,x_3,x_4\in\mathbb R^2\). Since the quantities
\[
\Delta_{ij}=1-(x_i^Tx_j)^2
\]
are unchanged if any \(x_j\) is replaced by \(-x_j\), we may choose the
signs of \(x_2,x_3,x_4\) so that
$x_1^Tx_j\ge 0$, $j=2,3,4$.
This does not change any of the \(\Delta_{ij}\).
Rotate the plane so that \(x_1=(1,0)\). Then we may write
\[
x_j=(\cos\theta_j,\sin\theta_j),
\,\,\,
\theta_j\in[-\pi/2,\pi/2],
\,\,\, j=2,3,4.
\]
Hence $\Delta_{1j}=\sin^2\theta_j$ and $\Delta_{ij}=\sin^2(\theta_i-\theta_j)$.

Order the three angles:
\[
\alpha\le \beta\le \gamma,
\]
and let the corresponding indices be \(p,q,r\), so that
\[
\theta_p=\alpha,\,\,\, \theta_q=\beta,\,\,\, \theta_r=\gamma.
\]
If \(\Delta_{12}=\Delta_{13}=\Delta_{14}=0\), we are in the exceptional
case. Otherwise, not all of \(\alpha, \beta, \gamma\) are zero. Thus either $\alpha < 0$ or $0 < \gamma$. 
We consider two cases depending on the sign of $\beta.$

If $\gamma \le 0$, then we must have $\alpha < 0$ and $\beta \le 0$.
Then $0\le \beta-\alpha\le -\alpha\le \pi/2$,
so
\[
|\sin(\beta-\alpha)|\le |\sin\alpha|.
\]
Since \(|\sin(\gamma-\alpha)|\le1\) and $ |\sin\alpha| > 0$, we conclude
\[
\frac{\Delta_{pq}\Delta_{pr}}{\Delta_{1p}}
=
\left(
\frac{|\sin(\beta-\alpha)|\,|\sin(\gamma-\alpha)|}
{|\sin\alpha|}
\right)^2
\le 1.
\]
If $0 \le \alpha$, then we must have $0\le \beta$ and $0< \gamma$.
Then $0\le \gamma-\beta\le \gamma\le \pi/2$,
so
\[
|\sin(\gamma-\beta)|\le \sin\gamma.
\]
Since \(|\sin(\gamma-\alpha)|\le1\) and $\sin\gamma > 0$, we have 
\[
\frac{\Delta_{pr}\Delta_{qr}}{\Delta_{1r}}
=
\left(
\frac{|\sin(\gamma-\alpha)|\,|\sin(\gamma-\beta)|}
{|\sin\gamma|}
\right)^2
\le 1.
\]
Thus one of the three well-defined ratios is at most \(1\).
\end{proof}

\bibliographystyle{amsplain}

\section{Statements \& Declarations}

The first author was partially supported by the Natural Sciences and Engineering Research Council (NSERC) of Canada. (Grant number RGPIN-2020-06425.)

The authors have no relevant financial or non-financial interests to disclose.

%
%
%
%

\end{document}